\documentclass[12pt, oneside,reqno]{amsart}
\usepackage{bbm}
\usepackage{mathtools}
\usepackage{amsmath, amsfonts, amssymb}
\usepackage{eucal}
\usepackage{mathrsfs}
\usepackage{latexsym}
\usepackage{cite}
\usepackage{xcolor}
\definecolor{refkey}{gray}{.45}
\definecolor{labelkey}{gray}{.45}

\usepackage{amsmath,amssymb}
\makeatletter
\newsavebox\myboxA
\newsavebox\myboxB
\newlength\mylenA

\newcommand*\xoverline[2][0.75]{%
    \sbox{\myboxA}{$\m@th#2$}%
    \setbox\myboxB\null% Phantom box
    \ht\myboxB=\ht\myboxA%
    \dp\myboxB=\dp\myboxA%
    \wd\myboxB=#1\wd\myboxA% Scale phantom
    \sbox\myboxB{$\m@th\overline{\copy\myboxB}$}%  Overlined phantom
    \setlength\mylenA{\the\wd\myboxA}%   calc width diff
    \addtolength\mylenA{-\the\wd\myboxB}%
    \ifdim\wd\myboxB<\wd\myboxA%
       \rlap{\hskip 0.5\mylenA\usebox\myboxB}{\usebox\myboxA}%
    \else
        \hskip -0.5\mylenA\rlap{\usebox\myboxA}{\hskip 0.5\mylenA\usebox\myboxB}%
    \fi}
\makeatother

\newcommand\wtilde[1]{\overset{\lower.4ex\hbox{$\scriptstyle \sim$}}{#1}}
\newcommand\wst[1]{\overset{\lower.5ex\hbox{$\scriptscriptstyle \sim$}}{#1}}
\newcommand{\blb}{\raise.3ex\hbox{$\scriptstyle \pmb \lbrack$}}
\newcommand{\sblb}{\raise.1ex\hbox{$\scriptscriptstyle \pmb \lbrack$}}
\newcommand{\brb}{\raise.3ex\hbox{$\scriptstyle \pmb \rbrack$}}
\newcommand{\sbrb}{\raise.1ex\hbox{$\scriptscriptstyle \pmb \rbrack$}}
\newcommand{\bla}{\raise.2ex\hbox{$\scriptstyle\pmb \langle$}}
\newcommand{\sbla}{\raise.1ex\hbox{$\scriptscriptstyle\pmb \langle$}}
\newcommand{\bra}{\raise.2ex\hbox{$\scriptstyle\pmb \rangle$}}
\newcommand{\sbra}{\raise.1ex\hbox{$\scriptscriptstyle\pmb \rangle$}}
\newcommand{\blrb}{\raise.3ex\hbox{$\scriptstyle \pmb | $}}
\newcommand{\sblrb}{\raise.1ex\hbox{$\scriptscriptstyle \pmb | $}}

\newcommand{\R}{\mathbb R}

\newcommand{\psum}{{+_{\negthinspace\kern-2pt p}}\,}
\newcommand{\qsum}[1]{{+_{\negthinspace\kern-2pt #1}}\,}
\newcommand{\dpsum}{{\tilde+_{\negthinspace\kern-1pt p}}\,}
\newcommand{\dqsum}[1]{{\tilde+_{\negthinspace\kern-1pt #1}}\,}
\newcommand{\lsub}[1]{\hskip -1.5pt\lower.5ex\hbox{$_{#1}$}}

\numberwithin{equation}{section}

\newtheorem{theo}{Theorem}[section]

\newtheorem{rema}[theo]{Remark}

\theoremstyle{definition}

\begin{document}

\title[]{A class of overdetermined problem for fractional Capacity}

\author[L. Qin]{Lei Qin }
\address{School of Mathematics,
Hunan University, Changsha, 410082, China}
\email{qlhnumath@hnu.edu.cn}

\author[L. Zhang]{Lu Zhang}
\address{School of Mathematics,
Hunan University, Changsha, 410082, China}
\email{luzhang@hnu.edu.cn}

%\thanks will become a 1st page footnote.
%\thanks{}

\maketitle

\baselineskip18pt

\parskip3pt

\begin{abstract}
In this paper, we consider an unconventional overdetermined problem through a property of concavity, which provides some characterizations of balls via Brunn-Minkowski inequalities. In this setting, our rsults can be viewed as the generalization of $p$-capacity in \cite{S15}, which have its own interest.
\end{abstract}

\section{Introduction}

In the field of partial differential equations, the overdetermined boundary value problem usually consists in a Dirichlet problem given in an unknown domain, whose solution is required to satisfy some extra conditions (classically a Neumann boundary condition) which are used to determine the shape of domain itself. After Serrin's initially groundbreaking work in\cite{S71}, the overdetermined elliptic problems were extensively studied, and they have a long history. From overdetermied problem of Laplacian to $p$-Laplace, a natural problem is to consider counterpart Serrin's corner results in nonlocal framework (or called fractional setting), which was studied by Fall and Jarohs \cite{FJ15}. They studied the non-local overdetermined problem, which needs new technique, since non-locality of the fractional Laplacian makes it difficult to study. There are many results in the literature dealing with fractional Lapacian. For example, the overdetermined problems of fractional Lapacian with fractional normal derivative constant prescribed have been considered in \cite{SV19}. Very recently, Ciraolo and Pollastro \cite{CP23} studied the overdetermined problem for fractional capacity, and they obtained some useful results, which will be used to prove our work.

An interesting problem is whether the geometry of domain is uniquely determined by imposing concavity. The unconventional overdetermined problem is given in \cite{S15} by adding extra condition, i.e., the convexity of a certain power of solution. Moreover, some other conditions are also given, for example, the unconventional overdetermined problem is given by extra condition, i.e., the complement to its maximum, which characterizes ellipsoid in \cite{HNST18}. Different from Serrin's work, which was based on moving plane method. The work depends on the concavity of Brunn-Minkowski inequality and differentiability properties of the shape functional by Fragal\`a \cite{F12}. Besides the overdetermined problem in classical Euclidean norm, many researchers also considered the anisotropic case in \cite{BCS16,BCS18,CP09}.

In this paper, we tackle the unconventional overdetermined problem in fractional framework.

Let $\Omega$ be a compact set, the fractional capacity of order $1/2$ (or $1$-Reisz capacity) of $\Omega$ is defined by
\[
\mathrm{Cap}_{1/2} (\Omega)= \mathrm{\inf} \{ \|\varphi\|_{\mathring{H}^{\frac{1}{2}}(\R^n)}^{2}: \ \varphi \in C_{c}^{\infty}(\R^n), \ \varphi\geq 1  \ \mathrm{on} \ \Omega\},
\]
where $\|\varphi\|_{\mathring{H}^{\frac{1}{2}}(\R^n)}^{2}$ is the $Gagliardo\ seminorm$, which is defined by
\[
\|\varphi\|_{\mathring{H}^{\frac{1}{2}}(\R^n)}:= \left( \int_{\R^n} \int_{\R^n} \frac{|\varphi(x)-\varphi(y)|^2}{|x-y|^{n+1}}dxdy  \right)^{1/2}.
\]
In fact, there exists unique fractional capacitary potential function $u$ such that
\begin{align*}
 \mathrm{Cap}_{1/2} (\Omega)=\int_{\R^n} \int_{\R^n} \frac{|u(x)-u(y)|^2}{|x-y|^{n+1}}dxdy=2c(n,1/2)\int_{\R^n} u(-\Delta)^{\frac{1}{2}}u dx,
\end{align*}
where
\[
c(n,1/2)=\int_{\R^n} \frac{1-\cos(\zeta_1)}{|\zeta|^{n+1}} d\zeta.
\]
Moreover, the function $u$ satisfies
\begin{equation}\label{PDE1}
\left\{
\begin{aligned}
& (-\Delta)^{\frac{1}{2}} u=0, \ \ \ \ \ \  \mathrm {on} \ \ \ \R^n \setminus \Omega,  \\
& u=1,  \ \ \ \ \ \  \ \ \ \ \ \ \ \ \  \mathrm {on} \ \ \ \Omega, \\
& {\lim}_{|x|\rightarrow +\infty} u(x)=0,
\end{aligned}
\right.
\end{equation}
where the $\frac{1}{2}$-fractional Laplacian $(-\Delta)^{\frac{1}{2}}$ is defined by
\begin{equation}\label{F-S-1}
(-\Delta)^{\frac{1}{2}}u(x)=c(n,1/2)^{-1}P.V.\int_{\R^{n}}\frac{u(x)-u(y)}{|x-y|^{n+1}}dy.
\end{equation}
For general fractional capacity and fractional Sobolev space, see details in \cite{COR22,CP23,NPV12}.

In what follows, for the convenience of describing our main results, we give some notations. A positive function $\upsilon$ is $\alpha$-concave, for some $\alpha\in [-\infty,+\infty)$, if
\[
\left\{
\begin{aligned}
& \upsilon^{\alpha}\ \mathrm{is\ concave,\ for\ } \alpha>0; \\
& \log \upsilon\ \mathrm{is\ concave,\ for\ } \alpha=0; \\
& \upsilon^{\alpha}\ \mathrm{is\ convex,\ for\ } \alpha<0; \\
& \text{all the super level sets}\ \{x\in \R^n:\ \upsilon(x)\geq t\}\ \text{are convex, for } \alpha=-\infty.
\end{aligned}
\right.
\]
Specially, if $\alpha=0$, $\upsilon$ is also called log-concave. If $\alpha=1$, $\upsilon$ is the usual concave function. The concavity index of $u$ is defined as follows:
\[
\alpha(\upsilon)=\sup \{\beta\leq 1:\ \upsilon\ \mathrm{is}\ \beta\mathrm{-concave} \}.
\]
It is obvious that the above supremum is in fact a maximum. The power index is optimal in the sense that no longer power ensures concavity for every convex set.

When $\Omega$ is convex and its capacitary function $u$ is level set convex, we define the $1/2$-fractional capacitary concavity index of $\Omega$ as follows:
\[
\alpha(\Omega)= \alpha(u).
\]
Indeed, when $\Omega$ is a ball of radius $R>0$ centered at $x_0$, it is easy to find explicitly the solution of (\ref{PDE1}), that is,
\[
u(x)=R^{n-1}|x-x_0|^{1-n}.
\]
It is obvious that $\mathrm{Cap}_{1/2}(B_{R}(x_0))=R^{n-1}$ and $u^{1/(1-n)}$ is convex, that is, $u$ is $1/(1-n)$-concave.

In this paper, we prove that the property of $u^{1/(1-n)}$ to be convex characterizes ball.
\begin{theo}\label{T1}
Let $\Omega$ be a bounded convex domain in $\R^n$. Then
\[
\alpha(\Omega)\leq \frac{1}{1-n}
\]
and equality holds if and only if $\Omega$ is a ball.
\end{theo}
To establish Theorem \ref{T1}, we shall need three basic facts:

$\bullet$ The Brunn-Minkowski inequality for fractional capacity $\mathrm{Cap}_{1/2}$ and its equality condition;

$\bullet$ $\mathrm{Cap}_{1/2}$ can be expressed by the behavior at infinity of the potential function $u$;

$\bullet$ The relationship between the fractional capacity $\mathrm{Cap}_{1/2}$ of level set of $u$ and $\Omega$.

We emphasize that these symbols are slightly different, such as the fractional capacity $\mathrm{Cap}_{1/2}$ is also called $1$-Riesz capacity in \cite{NR2015,MNR2018,MNR2022}. But it is essentially the same, and we can understand these properties in view of different directions.
\begin{rema}
\rm{ The Brunn-Minkowski inequality for fractional capacity $\mathrm{Cap}_{1/2}$ was established in  \cite[Theorem 1.1]{NR2015}. The equality case was obtained  for $n=2$ in \cite[Lemma 3.7]{MNR2018}. In fact, The equality case is also holds for any dimension. In Section \ref{Se4}, we give the equality case for $n\geq 3$ for the content integrity, and the approach is similar to the case $n=2$. }
\end{rema}

\begin{theo}\label{T2}
If the solution $u$ of $(\ref{PDE1})$ has two homothetic convex level sets, then $\Omega$ is a ball.
\end{theo}

In Theorem \ref{T1}, the overdetermination is given by the concavity property of solution $u$ of (\ref{PDE1}). While in Theorem \ref{T2}, the overdetermination is given by the existence of two homothetic level sets.

The organization of the paper is as follows. In Section \ref{Se2}, we introduce some basic facts and notions, which are useful for us to prove our main theorems. In Section \ref{Se3}, we prove the Theorem \ref{T2}. In Section \ref{Se4}, we prove  Theorem \ref{T1}.

\section{Preliminaries}\label{Se2}

In this section, we will concentrate on the basic facts about fractional capacity. For the current work, we also list some results, which comes from references.

Let $\mathbb{R}^n$ be the $n$-dimensional Euclidean space, and $\omega_n$ is the volume of unit ball. For $x\in \R^n$ and $R>0$, we denote by $B_{R}(x_0)$ the open ball of radius centered in $x$. In finite dimensional Euclidean space $\R^n$, a set $\Omega\subseteq \R^n$ is called a compact set if it is a closed, bounded set. We say that a set $\Omega\subseteq \R^n$ is a convex body if it is a compact convex subset of $\R^n$ with non-empty interior. Let $\mathcal{K}^n$ be the class of convex bodies in $\mathbb{R}^n$. Let $C_{c}^{\infty}(\R^n)$ be the set of functions from $C^{\infty}(\R^n)$ having compact support.

Let $\Omega_1, \Omega_2 \in \mathcal{K}^n $, then for $\lambda \in [0,1]$, the convex linear combination of $\Omega_1$ and $\Omega_2$ is defined by
\[
\lambda \Omega_1 +(1-\lambda)\Omega_2=\{\lambda x+(1-\lambda)y:\ x\in \Omega_1,\ y\in \Omega_2 \}.
\]

It is well known that if $\Omega$ is convex body, then solution $u$ of $(\ref{PDE1})$ is continuous and level set convex, that is,
\[
\Omega(t)=\{ x\in \R^n:\ u(x)\geq t \},\ \ \ t\in \R
\]
are convex (see \cite{NR2015}). Moreover, the potential function $u_t$ of $\Omega(t)$ is given by $u_t(x)=t^{-1}u(x)$, which is easily to be checked, it also holds that
\[
\mathrm{Cap}_{1/2}(\Omega(t))=t^{-1} \mathrm{Cap}_{1/2}(\Omega).
\]
Moreover, $\mathrm{Cap}_{1/2} (\Omega)$ can be characterized through he behaviour at infinity of the fractional capacity potential function $u$ of a convex body, that is,
\begin{equation}\label{B5}
\mathrm{Cap}_{1/2}(\Omega)= \lim\limits_{|x|\rightarrow \infty}  u(x) |x|^{n-1},
\end{equation}
(see Page 7 in \cite{NR2015}).

Given a compact set $\Omega\subseteq \R^n$, $n\geq 3$, the $p$-capacity is defined by
\[
\mathrm{Cap}_{p} (\Omega)= \min \left\{ \int_{\R^N} |\nabla \varphi|^p dx: \ \varphi \in C_{c}^{1}(\R^n,[0,1]), \ \varphi \geq \chi_\Omega   \right\}.
\]
when $p=2$, it is the Newtonian capacity. In fact, there exists unique capacitary function $u$ such that
\[
\mathrm{Cap}_{2} (\Omega)=\int_{\R^n \setminus \Omega} |\nabla u|^2 dx,
\]
and it satisfies the Euler-Lagrange equation:
\begin{equation}\nonumber
\left\{
\begin{aligned}
& \Delta u=0, \ \ \ \ \ \  \mathrm {on} \ \ \ \R^N \setminus \Omega,  \\
& u=1, \ \ \ \ \ \ \ \ \ \mathrm {on} \ \ \ \partial \Omega,\\
&{\lim}_{|x|\rightarrow +\infty} u(x)=0.
\end{aligned}
\right.
\end{equation}
It is well known that the following relation between the Newton capacity of a convex domain and the behavior at infinity of the Newtonian potential hold:
\begin{equation}\label{B6}
\mathrm{Cap}_{2} (\Omega)=(n-2)\omega_n \lim\limits_{|x|\rightarrow \infty} u(x)|x|^{n-2}.
\end{equation}
We stress that these results is also hold for $p$-capacity, for more details, refer to \cite{CS03}.

The relationship between the Newtonian capacity and $1/2$-fractional capacity is
\[
\text{Cap}_{1/2}(\Omega)=2\text{Cap}_{2}(\Omega \times\{0\}),
\]
where $\Omega \subseteq \R^2$ is a compact set, (cf.\cite[the formula (3.15)]{MNR2022}). We remark this conclusion is also hold for any dimension when $\Omega$ is convex domain, and we give its proof in Section \ref{Se3}, see formula (\ref{B4}).

\section{Proof of Theorem \ref{T2}}\label{Se3}

\begin{proof}[\bf Proof of Theorem \ref{T2}.]

By the assumption, let $0<r<s\leq 1$, $\Omega(r)$ and $\Omega(s)$ are the homothetic superlevel sets, that is, there exist $\rho>0$ and $\xi \in \R^n$ such that
\begin{equation}\label{4-1}
\Omega(r)=\rho \Omega(s)+\xi.
\end{equation}
Since $r<s$, it holds that
\[
\Omega(s)\subseteq \Omega(r),
\]
we have $\rho>1$. For $t\in (0,1]$, let us denote by $u_t$ the potential function of $\Omega(t)$, i.e., the solution of
\begin{equation}\label{PDE2}
\left\{
\begin{aligned}
& (-\Delta)^{\frac{1}{2}} u_t=0, \ \ \ \ \ \  \mathrm {on} \ \ \ \R^N \setminus \Omega(t),  \\
& \ u_t=1,  \ \ \ \ \ \  \ \ \ \ \ \ \ \  \mathrm {on} \ \ \ \partial \Omega(t), \\
& {\lim}_{|x|\rightarrow +\infty} u_t(x)=0.
\end{aligned}
\right.
\end{equation}
Then we have
\[
u_r(x)=\frac{u(x)}{r}, \ \ \ x\in \R^n \setminus \Omega(r)\ \ \ \mathrm{and}\ \ u_s(x)=\frac{u(x)}{s}, \ \ \ x\in \R^n \setminus \Omega(s).
\]
On the other hand, by (\ref{4-1}), we have
\[
u_r(x)=u_s \left(\frac{x-\xi}{\rho} \right).
\]
Thus, we deduce
\begin{equation}\label{4-2}
u(x)=\frac{r}{s} u \left(\frac{x-\xi}{\rho} \right), \ \ \ \ x\in \R^n\setminus \Omega(r).
\end{equation}
By the definition of $\mathrm{Cap}_{1/2}(\Omega)$, (\ref{B5}) and (\ref{4-2}), we have
\begin{align*}
\mathrm{Cap}_{1/2}(\Omega)&= \lim\limits_{|x|\rightarrow \infty}  u(x) |x|^{n-1}=\frac{r}{s} \lim\limits_{|x|\rightarrow \infty} u \left(\frac{x-\xi}{\rho} \right) |x|^{n-1}\\
&=\frac{r}{s} \rho^{n-1} \lim\limits_{|x|\rightarrow \infty} u \left(\frac{x-\xi}{\rho} \right) \left(\frac{|x-\xi|}{\rho} \right)^{n-1}  \left(\frac{|x|}{|x-\xi|} \right)^{n-1}\\
&=\frac{r}{s} \rho^{n-1} \mathrm{Cap}_{1/2}(\Omega).
\end{align*}
Then, we conclude that
\[
\frac{r}{s} \rho^{n-1}=1.
\]
Since $\rho>1$, we have $\frac{r}{s}<1$. This and (\ref{4-2}) imply that
\[
\Omega(t)=\rho \Omega(\frac{s}{r}t)+\xi, \ \ \ \ \text{for}\ \ t<r.
\]
Let
\[
s_0=s,\ \ s_1=r,\ \ s_k=(\frac{r}{s})^k s=\rho^{k(1-n)}s,\ \ k=2,3,\cdots,
\]
it is easily to seen that
\[
\mathop{\lim}_{k\rightarrow \infty} s_k=0.
\]
Moreover,
\begin{equation}\label{4-3}
\begin{aligned}
\Omega(s_k)&=\rho \Omega(s_{k-1})+\xi=\rho^2 \Omega(s_{k-2})+\rho \xi +\xi=\cdots\\
&=\rho^k \Omega(s_0)+\xi \mathop \sum_{i=0}^{k-1} \rho^i=\rho^k \Omega(s)+\xi \frac{\rho^k-1}{\rho-1}.
\end{aligned}
\end{equation}
If $x,y\in \partial \Omega(s)$, that is,
\[
u(x)=u(y)=s,
\]
and let
\[
x_k=\rho^k x+ \xi \frac{\rho^k-1}{\rho-1},
\]
\[
y_k=\rho^k y+ \xi \frac{\rho^k-1}{\rho-1}.
\]
Since $\rho>1$, we have
\[
\mathop {\lim}_{k\rightarrow \infty} |x_k|=\mathop {\lim}_{k\rightarrow \infty} |y_k|=\infty.
\]
Then, by (\ref{B5}), we get
\begin{equation}\label{4-4}
\mathop {\lim}_{k\rightarrow \infty}  u(x_k) |x_k|^{n-1}=\mathrm{Cap}_{1/2}(\Omega)=\mathop {\lim}_{k\rightarrow \infty}  u(y_k) |y_k|^{n-1}.
\end{equation}
By (\ref{4-3}), we have
\[
u(x_k)=u(y_k)=s_k.
\]
This and (\ref{4-4}) imply that
\[
\mathop {\lim}_{k\rightarrow \infty}  s_k |x_k|^{n-1}=\mathop {\lim}_{k\rightarrow \infty} s_k |y_k|^{n-1},
\]
that is,
\[
\mathop {\lim}_{k\rightarrow \infty} \rho^{k(1-n)}s \left| \rho^k x+ \xi \frac{\rho^k-1}{\rho-1} \right|^{n-1}=\mathop {\lim}_{k\rightarrow \infty} \rho^{k(1-n)}s \left| \rho^k y+ \xi \frac{\rho^k-1}{\rho-1} \right|^{n-1},
\]
this implies that
\[
\mathop {\lim}_{k\rightarrow \infty} \left|  x+ \xi \frac{1-\rho^{-k}}{\rho-1} \right|^{n-1}=\mathop {\lim}_{k\rightarrow \infty} \left|y+ \xi \frac{1-\rho^{-k}}{\rho-1} \right|^{n-1}.
\]
Since $\rho>1$, we can deduce that
\[
 \left|  x+ \xi \frac{1}{\rho-1} \right|= \left|y+ \xi \frac{1}{\rho-1} \right|=:R,
\]
which implies that $\Omega(s)$ is a ball with radius centered at the point $\xi/(1-\rho)$, that is,
\[
\Omega(s)=B(\frac{\xi}{1-\rho}, R),
\]
this and (\ref{4-3}) imply that
\[
\Omega(s_k)=B(\frac{\xi}{1-\rho}, R\rho^k).
\]
Then $u$ is radial in $\R^n \setminus \overline{\Omega(s)}$. By \cite[Theorem 1.1]{CP23}, we can deduce that $u$ is radial in $\R^n \setminus \overline{\Omega}$. Thus, $\Omega$ is a ball.
\end{proof}

\section{Proof of Theorem \ref{T1}}\label{Se4}

In this section, we chacracterize balls as unique domain under prescribed condition, i.e., we give the proof of Theorem \ref{T1}.

\begin{theo}\label{BMI}
Let $\Omega_1$ and $\Omega_2$ be convex bodies in $\R^n$ and for any $\lambda \in [0,1]$, it holds
\begin{equation}\label{B1}
\mathrm{Cap}_{1/2}(\lambda\Omega_1+(1-\lambda) \Omega_2)^{1/(n-1)} \geq\lambda \mathrm{Cap}_{1/2}( \Omega_1)^{1/(n-1)}+(1-\lambda)\mathrm{Cap}_{1/2}(\Omega_2)^{1/(n-1)}.
\end{equation}
Moreover, equality holds if and only if $K_{1}$ and $K_{2}$ are homothetic.
\end{theo}

\begin{proof}[\bf Proof.] By the result in \cite{NR2015}, we only to characterize the equality case. Let $\Omega_1$ and $\Omega_2$ be convex bodies in $\R^n$ such that equality holds in the Brunn-Minkowski inequation for some $\lambda\in (0,1)$, and set $\Omega_\lambda=\lambda \Omega_1+(1-\lambda) \Omega_2$. Let $u_{\Omega_1},u_{\Omega_2},u_{\Omega_\lambda}$ be the fractional capacitary potential functions of $\Omega_1,\Omega_2,\Omega_\lambda$, and let $U_1,U_2,U_\lambda$ be the harmonic extensions of  $u_{\Omega_1},u_{\Omega_2},u_{\Omega_\lambda}$ respectively, that is, $U_i$ ($i=1,2,\lambda$) are the solution of
\begin{equation}\nonumber
\left\{
\begin{aligned}
& (-\Delta)_{(x,t)} U_{i}=0, \ \ \ \ \ \  \mathrm {in} \ \ \ \R^{n+1}  \\
& U_i(x,0)=1,  \ \ \ \ \ \  \ \ \ \ \ \ x\in \Omega_i, \\
& {\lim}_{|(x,t)|\rightarrow +\infty} U_i(x,t)=0.
\end{aligned}
\right.
\end{equation}
According to the results in \cite{NR2015}, it is well-known that $U_i$ ($i=1,2,\lambda$) are continuous, capacitary functions of convex set $\Omega_i\times \{0\} \subseteq \R^{n+1}$ respectively if we omit a constant factor, which have positive capacity. Moreover,
\begin{equation}\label{B3}
U_i(x,0)=u_{\Omega_i}(x), \ \ \ \ \ i=1,2,\lambda,
\end{equation}
and $U_i$ is level set convex.  Given $r\in (0,1)$, $\Omega^r_i:=\{(x,t)\in \R^{n+1}:\ U_i\geq r\}$ ($i=1,2,\lambda$) is a convex set,
%and
%\begin{equation}\nonumber
%\{u_{\Omega_i}\geq r\}=\{(x,t)\in \R^{n+1}:\ U_i\geq r\} \cap \{t=0 \}
%\end{equation}
Since $U_i$ is continuous and $U_i(x,0)=1$, $x\in \Omega_i$, we can deduce that $\Omega^r_i$ has a postive volume. Thus,
\begin{equation}\label{B2}
\text{Cap}_{2}(\Omega^r_i)=\frac{1}{r}\text{Cap}_{2}(\Omega_i \times \{0 \}).\ \ \ \ \ \ \ i  =1,2, \lambda.
\end{equation}
By formula (\ref{B5}) and (\ref{B6}), we know that for $i=1,2,\lambda$, it holds
\begin{equation*}
\text{Cap}_{2}(\Omega_i\times\{0\})=\lim\limits_{|x|\rightarrow \infty}|x|^{n-1}U_i(x,0),
\end{equation*}
and
\begin{equation*}
\text{Cap}_{1/2}(\Omega_i)=\lim\limits_{|x|\rightarrow \infty}|x|^{n-1}u_{\Omega_i}(x).
\end{equation*}
From (\ref{B3}), we have
\begin{equation}\label{B4}
\text{Cap}_{1/2}(\Omega_i)=\text{Cap}_{2}(\Omega_i\times\{0\}), \ \ \ \ \ i=1,2,\lambda.
\end{equation}
Together with \eqref{B1}, \eqref{B2} and\eqref{B4}, we have
\begin{equation*}
\text{Cap}_{2}(\Omega^r_\lambda)^{1/(n-1)}
=\lambda\text{Cap}_{2}(\Omega^r_1)^{1/(n-1)}
+(1-\lambda)\text{Cap}_{2}(\Omega^r_2)^{1/(n-1)}.
\end{equation*}
By \cite[Theorem 1]{CS03}, we can deduce that $\Omega^r_1$ and $\Omega^r_2$ are homothetic for $r\in(0,1)$. Let $r\rightarrow 1$ we have $\Omega_1$ and $\Omega_2$ are homothetic.
\end{proof}

The following Theorem is the strengthened version of Theorem \ref{T1}, its proof is similar to Theorem \ref{T1}. In this paper, we only prove Theorem \ref{T1}.
\begin{theo}\label{T3}
Let $\Omega$ be a bounded domain in $\R^n$ (not necessarily convex) and $u$ be the solution of (\ref{PDE1}). If $u$ has three convex super level sets
\[
\Omega(r)=\{x\in \R^n:\ u(x)\geq r\}, \ \ \Omega(s)=\{x\in \R^n:\ u(x)\geq s\},
\]
\[
\Omega(t)=\{x\in \R^n:\ u(x)\geq t\},
\]
($0<s<r\leq 1$) such that
\[
t^{1/(1-n)}\geq [(1-\lambda)r^{1/(1-n)}+\lambda s^{1/(1-n)}],
\]
and
\[
\Omega(t)\supseteq (1-\lambda)\Omega(r)+\lambda \Omega(s),
\]
for some $\lambda \in (0,1)$, then $\Omega$ is a ball.
\end{theo}

\begin{proof}[\bf Proof of Theorem \ref{T1}.]
According to the case that $\Omega$ is a ball, we conclude that
\[
\alpha(\Omega)\leq \frac{1}{1-n}.
\]

Let $u$ be the solution of (\ref{PDE1}), assuming that $u^{1/(1-n)}$ is convex function in $\R^n$. By the monotonicity of the $\alpha$-concave function, that is, if $u$ is $\alpha$-concave, then for every $\beta\leq \alpha$, it is $\beta$-concave.
Let $\upsilon=u^{1/(1-n)}$, we know that $\upsilon$ is convex function in $\R^n$. Then, for any $\upsilon_0,\upsilon_1\in \R$ and $\mu\in (0,1)$, it holds that
\begin{equation}\label{B9}
\begin{aligned}
  \{x\in \R^n:\ \upsilon(x)\leq (1-\mu)\upsilon_0+\mu \upsilon_1 \} & \supseteq (1-\mu)\{ x\in \R^n: \upsilon(x)\leq \upsilon_0 \}\\
 &\quad +\mu \{ x\in \R^n:\ \upsilon(x)\leq \upsilon_1\}.
\end{aligned}
\end{equation}
Now, taking $r,s\in(0,1]$, fix $\lambda\in (0,1)$ and set
\[
t^{1/(1-n)}= (1-\lambda)r^{1/(1-n)}+\lambda s^{1/(1-n)}.
\]
By setting $\upsilon_0=r^{1/(1-n)}$ and $\upsilon_1=s^{1/(1-n)}$, we have
\begin{equation}\label{B10}
t^{1/(1-n)}=(1-\lambda)\upsilon_0+\lambda \upsilon_1,
\end{equation}
and
\begin{align*}
   & \Omega(r)=\{\upsilon\leq r^{1/(1-n)} \}, \\
   & \Omega(s)=\{\upsilon\leq s^{1/(1-n)} \}, \\
   & \Omega(t)=\{\upsilon\leq t^{1/(1-n)} \}.
\end{align*}
Then from (\ref{B9}), we have
\[
\Omega(t)\supseteq (1-\lambda)\Omega(r)+\lambda \Omega(s).
\]
By the monotonicity of fractional capacity with respect to set inclusion, we deduce
\[
\mathrm{Cap}_{1/2} [\Omega(t)]\geq \mathrm{Cap}_{1/2} [(1-\lambda)\Omega(r)+\lambda \Omega(s)].
\]
By the Brunn-Minkowski inequality for fractional capacity, i.e., Theorem \ref{BMI}, we obtain
\begin{equation}\label{B11}
\mathrm{Cap}_{1/2}(\Omega(t))^{1/(n-1)}\geq (1-\lambda) \mathrm{Cap}_{1/2}(\Omega(r))^{1/(n-1)}+\lambda\mathrm{Cap}_{1/2}( \Omega(s))^{1/(n-1)}.
\end{equation}
On the other hand, by
\begin{align*}
   & \mathrm{Cap}_{1/2}(\Omega(r))=r^{-1} \mathrm{Cap}_{1/2}(\Omega),\\
   & \mathrm{Cap}_{1/2}(\Omega(s))=s^{-1} \mathrm{Cap}_{1/2}(\Omega),
\end{align*}
and
\[
  \mathrm{Cap}_{1/2}(\Omega(t))=t^{-1} \mathrm{Cap}_{1/2}(\Omega).
\]
Then, combining with (\ref{B10}), (\ref{B11}), we have
\[
\mathrm{Cap}_{1/2}(\Omega(t))^{1/(n-1)}= (1-\lambda) \mathrm{Cap}_{1/2}(\Omega(r))^{1/(n-1)}+\lambda\mathrm{Cap}_{1/2}( \Omega(s))^{1/(n-1)}.
\]
Thus, the equality holds in the Brunn-Minkowski inequality for $\Omega(r)$ and $\Omega(s)$. By Theorem \ref{BMI}, we have $\Omega(r)$ and $\Omega(s)$ are homothetic. Thus, by Theorem \ref{T2}, $\Omega$ is a ball.
\end{proof}

\section*{Acknowledgement} The authors thank Prof. Yong Huang and Niufa Fang for their valued advices and comments on this subject.

\bibliographystyle{amsplain}

\end{document}